\documentclass[12pt,a4paper]{amsart}
 \pdfoutput=1
 
 \usepackage{amsmath,amssymb,amsthm}
 \usepackage{bm}  
 \usepackage{mathtools}
\usepackage{tikz}
\usetikzlibrary{arrows}
 \usepackage{thmtools}
 \usepackage{float} 
 
 \usepackage{xcolor} 	
 \usepackage{hyperref}
 \hypersetup{
 	colorlinks,
     linkcolor={red!60!black},
     citecolor={green!60!black},
     urlcolor={blue!60!black},
 }
 \usepackage[maxbibnames=99]{biblatex}
 \usepackage[utf8]{inputenc}
  \usepackage{csquotes}
 \usepackage[T1]{fontenc}
 \usepackage{lmodern}
 \usepackage[babel]{microtype}
 \usepackage[english]{babel}
 
 \usepackage{geometry}
 \usepackage{enumitem}

\theoremstyle{plain}
\newtheorem{theorem}{Theorem}[section]

\newtheorem{corollary}[theorem]{Corollary}
\newtheorem{lemma}[theorem]{Lemma}

\newtheorem{observation}[theorem]{Observation}

\newtheorem*{theorem*}{Theorem}
\newtheorem*{corollary*}{Corollary}

\theoremstyle{definition}

\title{Well-order a flame}

\author{Zsuzsanna Jank\'{o}}
\thanks{The first author was supported by the Hungarian Scientific Research Fund (NKFIH OTKA grant K143858), and the Hungarian Academy of 
Sciences under its Momentum Programme grant number LP2021-2. }
\address{Zsuzsanna Jank\'{o}, Corvinus University of Budapest, and ELTE Centre for Economic and Regional Studies}
\email{zsuzsanna.janko@uni-corvinus.hu}

\author{Attila Jo\'{o}}
\address{Attila Jo\'{o}, Department of Mathematics, Technion, Haifa, Israel 32000}
\address{Attila Jo\'{o},
Department of Mathematics, University of Hamburg, Bundesstra{\ss}e 55 (Geomatikum), 20146 Hamburg, Germany}
\email{a.joo@technion.ac.il}

\begin{document}

\begin{abstract}
An $r$-rooted (possibly infinite) digraph $ D=(V,E) $ is a flame if for every $ v\in V\setminus \{ r \} $ there exists a set of edge-disjoint paths from 
$r$ to $v$ in $D$ that covers all ingoing edges of $ v $. Flames were first studied by Lovász in his investigation of edge-minimal subgraphs of a 
rooted digraph that preserve all the local edge-connectivities from the root. He showed that these subgraphs are always flames.   Szeszlér later 
proved a common generalisation of Lovász' result and 
Edmonds’ disjoint arborescence theorem.  In this paper we focus on infinite flames and prove the following constructive characterisation. Every 
(possibly infinite) flame can be constructed transfinitely, starting from the empty edge set and adding a single edge at each step in such a way that 
every intermediate digraph is again a flame.  
\end{abstract}
\maketitle

\section{Introduction}
An $r$-rooted digraph is a (possibly infinite) digraph with a prescribed root vertex $r$ that has no ingoing edges. A flame is an $r$-rooted digraph 
$D=(V,E)$ such that for every $v\in V\setminus \{ r \}$ there exists a set of edge-disjoint paths from $r$ to $v$ that covers all ingoing edges of 
$v$. 
Recall that the local edge-connectivity $\lambda_D(u,v)$ is the maximal number of edge-disjoint paths from $u$ to $v$.
Lovász and 
Calvillo-Vives  independently investigated the edge-minimal subgraphs of finite rooted digraphs preserving all the local edge-connectivities from 
the 
root.   Let $D=(V,E)$ be a finite $r$-rooted digraph and suppose that  
$F\subseteq E$ is 
minimal with respect to the property that for the spanning 
subdigraph $D(F):=(V,F)$ we have $\lambda_{D(F)}(r,v)=\lambda_D(r,v)$ for each $v\in V\setminus \{ r \}$. Clearly, $F$ must contain at least 
$\lambda_D(r,v)$ ingoing edges of $v$ for $v\in V\setminus \{ r \}$. Lovász and 
Calvillo-Vives showed (see \cite[Theorem 2]{lovasz1973connectivity} and \cite{vives1987flames}) independently that $F$ contains 
\emph{exactly} $\lambda_D(r,v)$ ingoing edges of $v$ for each $v\in V\setminus \{ r \}$. In other 
words, every  spanning subdigraph that preserves all the local edge-connectivities from the root and is edge-minimal with respect to this property is 
a 
flame. Recall that an 
 arborescence is a directed rooted tree in which every vertex is reachable  from the root by a directed path. The union of vertex-disjoint 
 arborescences is called a branching. Szeszlér later proved 
the following strengthening of the theorem of Lovász and Calvillo-Vives:
\begin{theorem}[{\cite[Theorem 8]{szeszler2025some}}]\label{thm: szesz}
 Let $D=(V,E)$ be a finite $r$-rooted digraph and $m:=\max_{v\in V\setminus \{ r \}}\lambda_D(r,v)$. Then there are 
 $F_1 \subseteq F_2 \subseteq \cdots \subseteq F_m\subseteq E$ such that for every $i\in [m]$:  $D_i:=(V,F_i)$ is a flame with 
 $\lambda_{D_i}(r,v)=\min \{ 
 \lambda_{D}(r,v), i \}$ for $v\in V\setminus \{ r \}$
 and $F_{i}\setminus F_{i-1}$ forms a branching (where $F_0:=\emptyset$).
\end{theorem}
\noindent Note that the flame $D_m$ preserves all the local edge-connectivities from the root. Furthermore, if 
$k:=\min_{v\in 
 V\setminus \{ r 
\}}\lambda_D(r,v)$, then for every $i\in [k]$,  $F_{i}\setminus F_{i-1}$ must form a spanning arborescence. The latter is the classical 
`disjoint arborescence' theorem of Edmonds (see \cite{edmonds1973edge}). Theorem \ref{thm: szesz}, when applied to flames, allows us to 
construct flames as increasing chains of certain subflames. In this paper we prove that for every (possibly infinite) flame it is possible to find an 
increasing 
chain of 
subflames growing by a single edge in each step. More precisely, we prove the following:

\begin{theorem}\label{thm: main}
For every flame $D=(V,E)$, there exists a well-order 
$<$ on $ E $ of order type $ 
\left|E\right| $ such that for every subset $F\subseteq E$ that corresponds to an initial segment of $<$, the spanning subdigraph $  (V,F)  $ is 
a flame.
\end{theorem}

The special case of Theorem \ref{thm: main} where $E$ is finite follows from our previous result that the edge sets of the 
flame subgraphs of a finite rooted digraph form a greedoid (see \cite[Theorem 1.2]{jooGreedoidFlame2021}). The general case requires more 
sophisticated tools, mainly because an infinite flame may admit a proper subflame with exactly the same edge-connectivities from the root, which 
renders our key lemma \cite[Lemma 3.1]{jooGreedoidFlame2021} ineffective in the infinite setting.

The paper is organized as follows. In the following section we introduce our terminology. Section \ref{sec: prep} is devoted to the proof of the 
lemmas 
 we need. Finally, in Section \ref{sec: main} we prove Theorem \ref{thm: main}.
\section{Notation and terminology}

The digraphs $D=(V,E)$ considered in this paper may have parallel edges but no loops.
For $e\in E$, we denote by $D-e$ the digraph $(V,E\setminus\{e\})$.
If an edge $e$ is directed from $u$ to $v$, then $u$ is called the \emph{tail} and $v$ the \emph{head} of $e$.
For $X\subseteq V$, we denote by $\delta_D(X)$ the set of all edges $e\in E$ whose head lies in $X$ and whose tail lies in $V\setminus X$.

A \emph{rooted digraph} consists of a digraph $D=(V,E)$ together with a specified vertex $r\in V$, called the root, satisfying 
$\delta_D(r)=\emptyset$.
If a rooted digraph $D=(V,E)$ is fixed, then, to simplify notation, we identify any set $F\subseteq E$ with the spanning subdigraph $(V,F)$, which 
is considered to have the same root $r$.

If $U\subseteq V$ and $u\in U$, then the digraph obtained by \emph{contracting} the vertices of $U$ to $u$ is defined as follows.
Its vertex set is $(V\setminus U)\cup\{u\}$.
Edges whose head and tail both lie in $U$ are deleted.
Edges with tail in $U$ and head in $V\setminus U$ have $u$ as their new tail and keep their original head.
Edges with head in $U$ and tail in $V\setminus U$ have $u$ as their new head and keep their original tail.
Edges with neither endpoint in $U$ keep their original head and tail.
If $D$ is rooted, then the resulting digraph is rooted at the same vertex.
Moreover, if vertices are contracted to the root $r$, then any resulting ingoing edges of $r$ are deleted.

By a \emph{path} $P$ we always mean a finite directed path, and we identify paths with their edge sets.
An $XY$-path is a path that starts in $X$, terminates in $Y$, and has no internal vertices in $X\cup Y$.
The initial and terminal segments of a path $P$ are the subpaths of $P$ that share, respectively, the first and the last edge of $P$.
For singletons we omit brackets; for example, we write $\delta_D(v)$ instead of $\delta_D(\{v\})$ and $rv$-path instead of $\{r\}\{v\}$-path.

For a rooted digraph $D$, we define $\mathcal{G}_D(v)$ to be the collection of all sets $I\subseteq\delta_D(v)$ for which there exists a set 
$\mathcal{P}_I$ of edge-disjoint $rv$-paths in $D$ that covers $I$.
A rooted digraph $D=(V,E)$ is a \emph{flame} if $\delta_D(v)\in\mathcal{G}_D(v)$ for every $v\in V\setminus\{r\}$.
A \emph{transversal} of a set $\mathcal{P}$ of edge-disjoint paths is a set of edges containing exactly one edge from each path in $\mathcal{P}$.

For a nonempty family $\mathcal{F}$ of sets, we write $\bigcup\mathcal{F}$ and $\bigcap\mathcal{F}$ for the union and the intersection of the 
sets in $\mathcal{F}$, respectively. Let $[k]$ denote $\{ 1,2,\dots, k \}$ for $k\in \mathbb{N}$.
The variables $\alpha$, $\beta$, and $\gamma$ denote ordinal numbers.
The order type of a well-order is the unique ordinal $\alpha$ such that $(\alpha,\in)$ is isomorphic to it.
The cardinality $\lvert X\rvert$ of a set $X$ is the smallest ordinal $\alpha$ such that $X$ admits a well-order of order type $\alpha$.

\section{Preparations}\label{sec: prep}

Let an $r$-rooted digraph $D=(V,E)$ be fixed for this section.
A set $X\subseteq V\setminus\{r\}$ is called \emph{fillable} if there exists a set
\[
\mathcal{P}=\{P_e : e\in\delta_D(X)\}
\]
of edge-disjoint $rX$-paths in $D$ such that the last edge of $P_e$ is $e$ for every $e\in\delta_D(X)$.
Note that $D$ is a flame if and only if, for each $v\in V\setminus\{r\}$, the singleton $\{v\}$ is fillable.

\begin{lemma}\label{lem: intersection of fillable}
If $\mathcal{F}$ is a nonempty family of fillable sets, then $\bigcap\mathcal{F}$ is fillable.
\end{lemma}

\begin{proof}
Let $\mathcal{F}=\{X_\alpha : \alpha<\kappa\}$ and, for each $\alpha<\kappa$, let
$\{P_e^\alpha : e\in\delta_D(X_\alpha)\}$ be a path system witnessing that $X_\alpha$ is fillable, where $e$ is the last edge of $P_e^\alpha$.
Define $Y_\alpha:=\bigcap_{\beta<\alpha}X_\beta$ for $1\le \alpha\le \kappa$.

We construct by transfinite recursion a path system
$\{Q_e^\alpha : e\in\delta_D(Y_\alpha)\}$ witnessing that $Y_\alpha$ is fillable.
We maintain the following property: for every $\beta<\alpha$ and every $e\in\delta_D(Y_\alpha)$, the path $Q_e^\alpha$ contains exactly one edge
$f\in\delta_D(Y_\beta)$, and the initial segment of $Q_e^\alpha$ up to $f$ coincides with $Q_f^\beta$.

We set $Q_e^1:=P_e^0$ for each $e\in\delta_D(Y_1)=\delta_D(X_0)$.
Suppose that $\alpha\le\kappa$ and that $Q_e^\beta$ has already been defined for all $\beta<\alpha$ and all $e\in\delta_D(Y_\beta)$.
First assume that $\alpha=\beta+1$ is a successor ordinal.
For $e\in\delta_D(Y_{\beta+1})\cap\delta_D(Y_\beta)$, we define $Q_e^{\beta+1}:=Q_e^\beta$.
Now suppose that $e\in\delta_D(Y_{\beta+1})\setminus\delta_D(Y_\beta)$.
Then $e\in\delta_D(X_\beta)$ and the tail of $e$ lies in $Y_\beta$ (see Figure~\ref{fig: fillable}).
Let $f$ be the last edge of $P_e^\beta$ that lies in $\delta_D(Y_\beta)$.
We define $Q_e^{\beta+1}$ as the union of $Q_f^\beta$ and the terminal segment of $P_e^\beta$ starting with $f$. 
 
 \begin{figure}[h]
 \centering
 
\begin{tikzpicture}

\draw  (-3.5,3.5) rectangle (1,-1);
\draw  (-3.5,0.5) rectangle (0,-1);
\node[circle,inner sep=0pt,draw,minimum size=5] (v7) at (-5.5,4.5) {};
\node[circle,inner sep=0pt,draw,minimum size=5] (v1) at (-6.5,3.5) {};
\node[circle,inner sep=0pt,draw,minimum size=5] (v2) at (-6.5,2) {};
\node[circle,inner sep=0pt,draw,minimum size=5] (v3) at (-6,1) {};
\node[circle,inner sep=0pt,draw,minimum size=5] (v4) at (-5,0) {};
\node[circle,inner sep=0pt,draw,minimum size=5] (v5) at (-4,-0.5) {};
\node[circle,inner sep=0pt,draw,minimum size=5] (v6) at (-2.5,-0.5) {};

\node[circle,inner sep=0pt,draw,minimum size=5] (v8) at (-4,5.5) {};
\node[circle,inner sep=0pt,draw,minimum size=5] (v9) at (-3,5.5) {};
\node[circle,inner sep=0pt,draw,minimum size=5] (v10) at (-2,5) {};
\node[circle,inner sep=0pt,draw,minimum size=5] (v11) at (-1.5,4) {};
\node[circle,inner sep=0pt,draw,minimum size=5] (v12) at (-1.5,3) {};
\node[circle,inner sep=0pt,draw,minimum size=5] (v13) at (-5,2.5) {};
\node[circle,inner sep=0pt,draw,minimum size=5] (v14) at (-4,2.5) {};
\node[circle,inner sep=0pt,draw,minimum size=5] (v15) at (-2.5,2.5) {};

\draw[-triangle 60]   (v1) edge (v2);
\draw[-triangle 60]   (v2) edge (v3);
\draw[-triangle 60]   (v3) edge (v4);
\draw[-triangle 60]   (v4) edge (v5);
\draw[-triangle 60]   (v5) edge (v6);
\draw[-triangle 60]   (v1) edge (v7);
\draw[-triangle 60]   (v7) edge (v8);
\draw[-triangle 60]   (v8) edge (v9);
\draw[-triangle 60]   (v9) edge (v10);
\draw[-triangle 60]   (v10) edge (v11);
\draw[-triangle 60]   (v11) edge (v12);
\draw[-triangle 60]   (v1) edge (v13);
\draw[-triangle 60]   (v13) edge (v14);
\draw[-triangle 60]   (v14) edge (v15);

\draw[-triangle 60, dashed]  (v11) edge (v12);
\draw[-triangle 60, dashed]  (v14) edge (v15);

\node[circle,inner sep=0pt,draw,minimum size=5] (v16) at (-2,1.5) {};
\node[circle,inner sep=0pt,draw,minimum size=5] (v17) at (-1,1) {};
\node[circle,inner sep=0pt,draw,minimum size=5] (v18) at (-1,0) {};
\draw[-triangle 60, dashed]  (v15) edge (v16);
\draw[-triangle 60, dashed]  (v16) edge (v17);
\draw[-triangle 60, dashed]  (v17) edge (v18);

\node at (-6.8,3.4) {$r$};
\node at (-3,0.8) {$Y_{\beta+1}$};
\node at (-3,3.8) {$Y_{\beta}$};
\node at (-3.2,2.7) {$f$};
\node at (-4.6,2.8) {$Q_f^{\beta}$};
\node at (-1.2,0.7) {$e$};
\node at (-2.6,1.8) {$P_e^{\beta}$};
\end{tikzpicture}
 \caption{The construction of the paths $Q_e^{\beta+1}$ for $e\in \delta_D(Y_{\beta+1})$.} \label{fig: fillable}
 \end{figure}
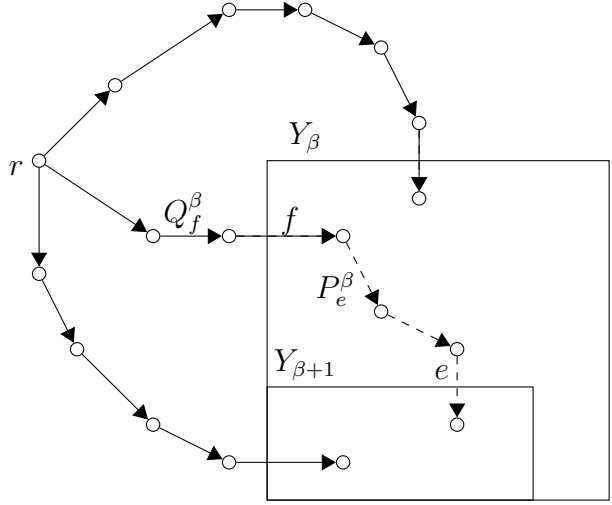
 
 Suppose now that $\alpha$ is a limit ordinal.
 For each $e\in\delta_D(Y_\alpha)$, there exists a smallest $\beta<\alpha$ such that $e\in\delta_D(Y_\beta)$.
 We define $Q_e^\alpha:=Q_e^\beta$.
 This completes the transfinite construction.
 The path system $\{Q_e^\kappa : e\in\delta_D(Y_\kappa)\}$ witnesses that $\bigcap\mathcal{F}$ is fillable.
 \end{proof}
 
 \begin{corollary}\label{cor: fillable closure}
 For every $X\subseteq V\setminus\{r\}$, there exists a $\subseteq$-smallest fillable set $\mathsf{fill}_D(X)$ containing $X$.
 \end{corollary}
 
 \begin{proof}
 The set $V\setminus\{r\}$ is fillable and includes $X$.
 Therefore, by Lemma~\ref{lem: intersection of fillable}, the intersection of all fillable supersets of $X$ is the desired set $\mathsf{fill}_D(X)$.
 \end{proof}
 
 Suppose that $v\in V\setminus\{r\}$ and that $\delta_D(v)\in\mathcal{G}_D(v)$.
 A set $T\subseteq V\setminus\{r\}$ is called \emph{$v$-tight} in $D$ if, for every path system $\mathcal{P}$ witnessing
 $\delta_D(v)\in\mathcal{G}_D(v)$, the set $\delta_D(T)$ is a transversal of $\mathcal{P}$.
 
 \begin{observation}\label{obs: tight is fillable}
 If a set $T$ is $v$-tight, then $T$ is fillable.
 \end{observation}
 
 \begin{proof}
 Let $\mathcal{P}$ be a path system witnessing $\delta_D(v)\in\mathcal{G}_D(v)$.
 Then the initial segments of the paths in $\mathcal{P}$ up to their unique edges in $\delta_D(T)$ witness that $T$ is fillable.
 \end{proof}
 
 \begin{lemma}\label{lem: tight union}
 If $\mathcal{F}$ is a nonempty family of $v$-tight sets, then both $\bigcup\mathcal{F}$ and $\bigcap\mathcal{F}$ are $v$-tight.
 \end{lemma}
 
 \begin{proof}
 We prove the tightness of $\bigcup\mathcal{F}$; the argument for $\bigcap\mathcal{F}$ is analogous.
 Let $\mathcal{P}$ be a path system witnessing $\delta_D(v)\in\mathcal{G}_D(v)$ and let $e\in\delta_D(\bigcup\mathcal{F})$.
 Then there exists $T\in\mathcal{F}$ such that $e\in\delta_D(T)$.
 By the tightness of $T$, the edge $e$ lies on some path $P\in\mathcal{P}$.
 It remains to show that no path in $\mathcal{P}$ contains more than one edge from $\delta_D(\bigcup\mathcal{F})$.
 Suppose to the contrary that some $Q\in\mathcal{P}$ does.
 Then $Q$ enters $\bigcup\mathcal{F}$, leaves it, and reenters.
 Hence there exists $T'\in\mathcal{F}$ such that $Q$ enters and later leaves $T'$.
 Since $Q$ terminates at $v\in\bigcap\mathcal{F}$, it must reenter $T'$,
 implying that $Q$ contains more than one edge of $\delta_D(T')$, contradicting the tightness of $T'$.
 \end{proof}
 
 \begin{corollary}\label{cor: largest tight}
 If $\delta_D(v)\in\mathcal{G}_D(v)$, then there exists a $\subseteq$-largest $v$-tight set $T_{v,D}$.
 \end{corollary}
 
 \begin{proof}
 The singleton $\{v\}$ is clearly $v$-tight.
 Thus the family of all $v$-tight sets is nonempty, and its union is the $\subseteq$-largest $v$-tight set.
 \end{proof}

We require the following classical lemma, originally proved in the context of vertex-disjoint paths in~\cite{bohme2001menger}.

\begin{lemma}[Augmenting path lemma]\label{lem: Augmenting path lemma}
Let $\mathcal{P}$ be a system of edge-disjoint $rv$-paths in $D$, where $v\in V\setminus\{r\}$.
Then exactly one of the following possibilities occurs:
\begin{enumerate}[label=(\roman*)]
 \item\label{item: first possib}
 There exists a system $\mathcal{Q}$ of edge-disjoint $rv$-paths such that
 \[
 \lvert \mathcal{Q}\setminus\mathcal{P}\rvert
 =
 \lvert \mathcal{P}\setminus\mathcal{Q}\rvert + 1 \in \mathbb{N},
 \]
 and $\mathcal{Q}$ covers all the outgoing edges of $r$ and ingoing edges of $v$ that are covered by $\mathcal{P}$.
 \item\label{item: second possib}
 There exists a set $X\subseteq V\setminus\{r\}$ containing $v$ such that $\delta_D(X)$ is a transversal of $\mathcal{P}$.
\end{enumerate}
\end{lemma}

\begin{proof}
Assume that \ref{item: second possib} holds, and let $\mathcal{P}'\subseteq\mathcal{P}$ with
$\lvert \mathcal{P}\setminus\mathcal{P}'\rvert = k \in \mathbb{N}$.
Then exactly  $k$ ingoing edges of $X$  are not used by the paths in $\mathcal{P}'$.
Since $\delta_D(X)$ meets every $rv$-path, it is impossible to add $k+1$ further $rv$-paths to $\mathcal{P}'$ in such a way that the resulting path 
system remains edge-disjoint.
Therefore, \ref{item: first possib} cannot hold.

To show that at least one of the two possibilities occurs, let $D'$ be the auxiliary digraph obtained from $D$ by reversing all edges in 
$\bigcup\mathcal{P}$.
Suppose that there exists an $rv$-path $P$ in $D'$ that uses reversed edges from $k$ paths in $\mathcal{P}$. Then the symmetric difference of $P$ 
with these 
paths can be decomposed into $k+1$ edge-disjoint $rv$-paths in $D$ covering those outgoing edges of $r$ and ingoing edges of $v$ that are 
covered by these $k$ paths.
Together with the untouched paths of $\mathcal{P}$, this yields a system $\mathcal{Q}$ as required in \ref{item: first possib}.
Otherwise, the set $X$ of vertices that are not reachable from $r$ by a path in $D'$ satisfies \ref{item: second possib}.
\end{proof}

\begin{corollary}\label{cor: n more max}
Suppose that $v\in V\setminus\{r\}$ and that $\delta_D(v)\in\mathcal{G}_D(v)$.
Let $\mathcal{P}$ be a system of edge-disjoint $rv$-paths that covers all but $n\in\mathbb{N}$ edges of $\delta_D(v)$.
Then there exists a system $\mathcal{Q}$ witnessing $\delta_D(v)\in\mathcal{G}_D(v)$ such that $\mathcal{Q}$ uses all outgoing edges of $r$ 
that are used by $\mathcal{P}$ and at most $n$ more.
\end{corollary}

\begin{proof}
We proceed by induction on $n$.
For $n=0$, the system $\mathcal{Q}:=\mathcal{P}$ is suitable.
Assume that the statement holds for $n$, and let $\mathcal{P}$ be a system of edge-disjoint $rv$-paths that covers all but $n+1$ edges of 
$\delta_D(v)$.

Apply Lemma~\ref{lem: Augmenting path lemma} to $\mathcal{P}$.
If case~\ref{item: first possib} occurs, then the resulting system $\mathcal{P}'$ must cover one additional outgoing edge of $r$ and one additional 
ingoing edge of $v$.
Applying the induction hypothesis to $\mathcal{P}'$ yields the desired system $\mathcal{Q}$.

Suppose instead that case~\ref{item: second possib} occurs, and let $X$ be the corresponding set.
Let $\mathcal{R}$ be a path system witnessing $\delta_D(v)\in\mathcal{G}_D(v)$, and let $\mathcal{R}'$ consist of the terminal segments of the 
paths in $\mathcal{R}$ starting from their last edges in $\delta_D(X)$.
For each $R\in\mathcal{R}'$, there exists a unique path $P_R\in\mathcal{P}$ containing the first edge of $R$.
Let $Q_R$ be the $rv$-path obtained by concatenating $R$ with the initial segment of $P_R$ up to its unique edge in $\delta_D(X)$.
Then
\[
\mathcal{Q}:=\{Q_R : R\in\mathcal{R}'\}
\]
is a system of edge-disjoint paths witnessing $\delta_D(v)\in\mathcal{G}_D(v)$ and using the same outgoing edges of $r$ as $\mathcal{P}$.
\end{proof}

\begin{corollary}\label{cor: n more max extra}
Let $U$ be a fillable set containing $v\in V\setminus\{r\}$, and let $\mathcal{P}$ be a system of edge-disjoint $rv$-paths that covers all but 
$n\in\mathbb{N}$ edges of $\delta_D(v)$.
Let $\mathcal{P}'$ consist of the terminal segments of the paths in $\mathcal{P}$ starting from their last edges in $\delta_D(U)$.
Then there exists a system $\mathcal{Q}$ witnessing $\delta_D(v)\in\mathcal{G}_D(v)$ such that $\mathcal{Q}$ uses all edges of $\delta_D(U)$ 
used by $\mathcal{P}'$ and at most $n$ additional edges.
\end{corollary}

\begin{proof}
Let $D'=(U\cup\{r\},E')$ be the auxiliary digraph obtained from $D$ by contracting $V\setminus U$ to $r$.
Apply Corollary~\ref{cor: n more max} to the path system $\mathcal{P}'$  in $D'$.
Finally, extend the resulting paths backward to $rv$-paths of $D$ using the fact that $U$ is fillable in $D$.
\end{proof}

\begin{corollary}\label{cor: at most n}
Let $T$ be a $v$-tight set, and let $\mathcal{P}$ be a system of edge-disjoint $rv$-paths that covers all but $n\in\mathbb{N}$ edges of 
$\delta_D(v)$.
Let $\mathcal{P}'$ consist of the terminal segments of the paths in $\mathcal{P}$ starting from their last edges in $\delta_D(T)$.
Then at most $n$ edges of $\delta_D(T)$ are not used by the paths in $\mathcal{P}'$.
\end{corollary}

\begin{proof}
Suppose, for a contradiction, that more than $n$ edges of $\delta_D(T)$ are not used by the paths in $\mathcal{P}'$.
Since tight sets are fillable (see Observation~\ref{obs: tight is fillable}), we may apply Corollary~\ref{cor: n more max extra} with $T$ in the role 
of $U$.
The resulting system $\mathcal{Q}$ then fails to use all edges of $\delta_D(T)$, contradicting the assumption that $T$ is $v$-tight.
\end{proof}

\begin{lemma}\label{lem: char no single}
Let $v\in V\setminus\{r\}$ and let $e\in\delta_D(v)$ be such that
$\delta_D(v)\setminus\{e\}\in\mathcal{G}_D(v)$.
Then $\delta_D(v)\notin\mathcal{G}_D(v)$ if and only if the tail $u$ of $e$ belongs to $T_{v,D-e}$.
\end{lemma}

\begin{proof}
Assume first that $u\in T_{v,D-e}=:T$ and suppose, for a contradiction, that $\delta_D(v)\in\mathcal{G}_D(v)$ is witnessed by a path system 
$\mathcal{P}$.
Let $P_e$ be the unique path in $\mathcal{P}$ whose last edge is $e$.
Since $T$ is $v$-tight in $D-e$, the set $\delta_{D-e}(T)$ is a transversal of $\mathcal{P}\setminus\{P_e\}$.
Clearly, $\delta_{D-e}(T)=\delta_D(T)$ because $u\in T$.
As $P_e$ is an $rv$-path, it contains an edge $f\in\delta_D(T)$.
However, since $f$ is also used by some path in $\mathcal{P}\setminus\{P_e\}$, the system $\mathcal{P}$ is not edge-disjoint, a contradiction.

For the converse, assume that $\delta_D(v)\notin\mathcal{G}_D(v)$.
It suffices to show that the set $T:=\mathsf{fill}_D(\{u,v\})$ (see Corollary \ref{cor: fillable closure}) is $v$-tight in $D-e$.
Suppose, for a contradiction, that it is not.
Then there exists a path system $\mathcal{P}$ in $D-e$ witnessing
$\delta_D(v)\setminus\{e\}\in\mathcal{G}_D(v)$ such that $\delta_{D-e}(T)=\delta_D(T)$ is not a transversal of $\mathcal{P}$.
Let $\mathcal{P}'$ consist of the terminal segments of the paths in $\mathcal{P}$ starting from their last edges in $\delta_D(T)$.
By the choice of $\mathcal{P}$, the initial edges of the paths in $\mathcal{P}'$ form a proper subset of $\delta_D(T)$ (see Figure~\ref{fig: cannot 
add e}).

Let $D'=(T\cup\{r\},E')$ be the auxiliary digraph obtained from $D$ by contracting $V\setminus T$ to $r$.
Then $\mathcal{P}'$ is a system of edge-disjoint $rv$-paths in $D'$ covering all ingoing edges of $v$ except $e$.
We apply Lemma~\ref{lem: Augmenting path lemma} in $D'$ to $\mathcal{P}'$.
   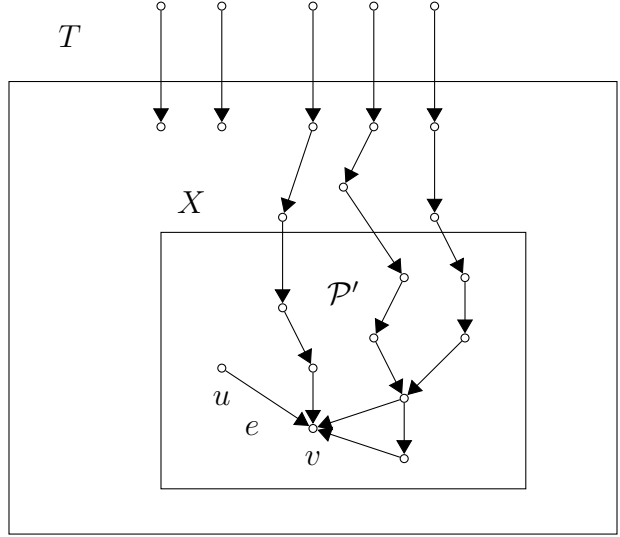
\begin{figure}[h]
     \centering
     
   \begin{tikzpicture}[scale=2]
   
   \draw  (-2,2.5) rectangle (2,-0.5);
   \node[circle,inner sep=0pt,draw,minimum size=3] (v1) at (-0.6,0.6) {};
   \node[circle,inner sep=0pt,draw,minimum size=3] (v2) at (0,0.2) {};
   \draw[-triangle 60]   (v1) edge (v2);

   \node at (-0.6,0.4) {$u$};
   \node at (0,0) {$v$};

   \node[circle,inner sep=0pt,draw,minimum size=3] (v3) at (-1,3) {};
   \node[circle,inner sep=0pt,draw,minimum size=3] (v4) at (-1,2.2) {};
   \node[circle,inner sep=0pt,draw,minimum size=3] (v5) at (-0.6,3) {};
   \node[circle,inner sep=0pt,draw,minimum size=3] (v6) at (-0.6,2.2) {};
   \draw[-triangle 60]  (v3) edge (v4);
   \draw[-triangle 60]  (v5) edge (v6);
   \node[circle,inner sep=0pt,draw,minimum size=3] (v7) at (0,3) {};
   \node[circle,inner sep=0pt,draw,minimum size=3] (v8) at (0,2.2) {};
   \node[circle,inner sep=0pt,draw,minimum size=3] (v9) at (-0.2,1.6) {};
   \node[circle,inner sep=0pt,draw,minimum size=3] (v10) at (-0.2,1) {};
   \node[circle,inner sep=0pt,draw,minimum size=3] (v11) at (0,0.6) {};
   
   \draw[-triangle 60]   (v7) edge (v8);
   \draw[-triangle 60]   (v8) edge (v9);
   \draw[-triangle 60]   (v9) edge (v10);
   \draw[-triangle 60]   (v10) edge (v11);
   \draw[-triangle 60]   (v11) edge (v2);
   \node[circle,inner sep=0pt,draw,minimum size=3] (v12) at (0.4,3) {};
   \node[circle,inner sep=0pt,draw,minimum size=3] (v13) at (0.4,2.2) {};
   \node[circle,inner sep=0pt,draw,minimum size=3] (v18) at (0.8,3) {};
   \node[circle,inner sep=0pt,draw,minimum size=3] (v19) at (0.8,2.2) {};

   \node[circle,inner sep=0pt,draw,minimum size=3] (v14) at (0.2,1.8) {};
   \node[circle,inner sep=0pt,draw,minimum size=3] (v15) at (0.6,1.2) {};
   \node[circle,inner sep=0pt,draw,minimum size=3] (v16) at (0.4,0.8) {};
   \node[circle,inner sep=0pt,draw,minimum size=3] (v17) at (0.6,0.4) {};
   
   \draw[-triangle 60]   (v12) edge (v13);
   \draw[-triangle 60]   (v13) edge (v14);
   \draw[-triangle 60]   (v14) edge (v15);
   \draw[-triangle 60]   (v15) edge (v16);
   \draw[-triangle 60]   (v16) edge (v17);
   \draw[-triangle 60]   (v17) edge (v2);
   \node[circle,inner sep=0pt,draw,minimum size=3]  (v20) at (0.8,1.6) {};
   \node[circle,inner sep=0pt,draw,minimum size=3] (v21) at (1,1.2) {};
   \node[circle,inner sep=0pt,draw,minimum size=3] (v22) at (1,0.8) {};

   \node[circle,inner sep=0pt,draw,minimum size=3] (v23) at (0.6,0) {};
   \draw[-triangle 60]   (v18) edge (v19);
   \draw [-triangle 60]  (v19) edge (v20);
   \draw[-triangle 60]   (v20) edge (v21);
   \draw[-triangle 60]   (v21) edge (v22);
   \draw[-triangle 60]   (v22) edge (v17);
   \draw[-triangle 60]   (v17) edge (v23);
   \draw[-triangle 60]   (v23) edge (v2);
   \draw  (-1,1.5) rectangle (1.4,-0.2);
   \node at (-0.8,1.7) {$X$};
   \node at (-1.6,2.8) {$T$};
   \node at (-0.4,0.2) {$e$};
   \node at (0.2,1.1) {$\mathcal{P}'$};
   \end{tikzpicture}
     \caption{The tightness of $T$ in Lemma \ref{lem: char no single}} \label{fig: cannot add e}
     \end{figure}

  If case~\ref{item: first possib} occurs, then the resulting path system witnesses that
  $\delta_{D'}(v)\in \mathcal{G}_{D'}(v)$. Hence, by the construction of $D'$ and the
  fillability of $T$, we conclude that $\delta_{D}(v)\in \mathcal{G}_{D}(v)$, which is a
  contradiction.
  
  Let $X$ be as in case~\ref{item: second possib}. Then $\delta_{D'}(X)$ is a transversal of
  $\mathcal{P}'$. Since $e\notin \bigcup \mathcal{P}'$, we must have
  $e\notin \delta_{D'}(X)$, and thus $u\in X$. The initial segments of the paths in
  $\mathcal{P}'$ up to their first edges in $X$ witness that $X$ is fillable in $D'$.
  Since $T\supseteq X$ is fillable in $D$, it follows that $X$ is fillable in $D$ as well.
  Finally, $X$ must be a proper subset of $T$, because the heads of ingoing edges of $T$
  that are not used by $\mathcal{P}'$ cannot lie in $X$. This contradicts the fact that $T$
  is the smallest fillable set containing $\{u,v\}$.
 
\end{proof}
\begin{corollary}\label{cor: superfillable}
If $v\in V\setminus\{r\}$ and $e\in \delta_D(T_{v,D-e})$, then $T_{v,D-e}$ is fillable in
$D$ (and not only in $D-e$).
\end{corollary}

\begin{proof}
Suppose, for a contradiction, that this is not the case. Let $D'$ be the auxiliary
digraph obtained from $D$ by contracting $T_{v,D-e}=:T$ to $v$. To obtain a contradiction,
we will apply Lemma~\ref{lem: char no single} to find a strictly larger $v$-tight set,
contradicting the maximality of $T_{v,D-e}$.

By construction, $\delta_{D'-e}(v)=\delta_{D-e}(T)$. Since $T_{v,D-e}$ is fillable in
$D-e$ by Observation~\ref{obs: tight is fillable}, we have
$\delta_{D'-e}(v)\in \mathcal{G}_{D'-e}(v)$. By our assumption,
$\delta_{D'}(v)\notin \mathcal{G}_{D'}(v)$, and therefore
Lemma~\ref{lem: char no single} yields a $v$-tight set $T'$ in $D'-e$ that contains the
tail $u$ of $e$.

Let $T''$ be the vertex set in $D$ corresponding to $T'$, obtained by expanding
the vertex $v\in T'$ to the set $T$. Then $T''\supsetneq T$, since
$u\in T''\setminus T$. To obtain a contradiction, we show that $T''$ is tight in $D-e$.

Let $\mathcal{P}$ be a path system in $D-e$ witnessing
$\delta_{D-e}(v)\in \mathcal{G}_{D-e}(v)$. Since $T$ is $v$-tight in $D-e$, the set
$\delta_{D-e}(T)=\delta_{D'-e}(v)$ is a transversal of $\mathcal{P}$. Let $\mathcal{P}'$
consist of the initial segments of the paths in $\mathcal{P}$ up to their first edges in
$\delta_{D-e}(T)$. Then $\mathcal{P}'$ is a system of edge-disjoint paths in $D'-e$
witnessing $\delta_{D'-e}(v)\in \mathcal{G}_{D'-e}(v)$.

Since $T'$ is tight in $D'-e$, the set $\delta_{D'-e}(T')$ is a transversal of
$\mathcal{P}'$. As $\delta_{D'-e}(T')=\delta_{D-e}(T'')$, it follows that
$\delta_{D-e}(T'')$ is a transversal of $\mathcal{P}$. Hence $T''$ is tight in $D-e$,
which completes the proof.
\end{proof}

\begin{corollary}\label{cor: insert one}
If $F\subseteq E$ is a flame, $v\in V\setminus\{r\}$, and $e\in E\setminus F$ such that
$e\in \delta_D(T_{v,F})$, then $F\cup\{e\}$ is a flame.
\end{corollary}

\begin{proof}
Let $v$ be the head of $e$. By Corollary~\ref{cor: superfillable} (applied to
$F\cup\{e\}$), the set $T_{v,F}$ is fillable in $F\cup\{e\}$. Fix a path system
$\mathcal{P}$ witnessing this, and let $\mathcal{Q}$ be a path system witnessing
$\delta_F(v)\in \mathcal{G}_F(v)$. Let $\mathcal{Q}'$ consist of the terminal segments of
the paths in $\mathcal{Q}$ starting from their last edges in $\delta_F(T)$.

For each $Q\in \mathcal{Q}'$, unite $Q$ with the unique path $P_Q\in \mathcal{P}$ whose
last edge is the first edge of $Q$, and also include the unique path $P_e\in\mathcal{P}$
whose last edge is $e$. The resulting path system witnesses that
$\delta_{F\cup\{e\}}(v)\in \mathcal{G}_{F\cup\{e\}}(v)$.
\end{proof}

\section{Proof of Theorem~\ref{thm: main}}\label{sec: main}

Assume that $D=(V,E)$ is a flame. For each $v\in V\setminus\{r\}$, let
$\mathcal{P}_v=\{P_e : e\in \delta_D(v)\}$
be a path system witnessing $\delta_D(v)\in \mathcal{G}_D(v)$, and let
$\mathcal{P}:=\bigcup_{v\in V\setminus\{r\}}\mathcal{P}_v$. Note that for each $e\in E$ there is a unique path $P_e\in\mathcal{P}$ whose last 
edge is $e$. A set $F\subseteq E$ is called \emph{$\mathcal{P}$-respecting} if $P_e\subseteq F$ for
every $e\in F$. Every $\mathcal{P}$-respecting set $F$ is a flame, since for each
$v\in V\setminus\{r\}$ the paths $\{P_e : e\in \delta_F(v)\}$  lie in $F$ and witness $\delta_F(v)\in \mathcal{G}_F(v)$.

\begin{observation}
The union of a $\subseteq$-chain of $\mathcal{P}$-respecting flames is again a
$\mathcal{P}$-respecting flame.
\end{observation}
We call a set $F\subseteq E$ \emph{almost $\mathcal{P}$-respecting} if there are only
finitely many edges $e\in F$ such that $P_e\not\subseteq F$. An almost
$\mathcal{P}$-respecting set is not necessarily a flame.

\begin{lemma}\label{lem: add fin}
If $F$ is an almost $\mathcal{P}$-respecting flame and $e\in E\setminus F$, then there
exists a finite set $\{e_i : i\leq n\}\subseteq E\setminus F$ containing $e$ such that
$F\cup\{e_i : i\leq m\}$ is a flame for every $m\leq n$.
\end{lemma}

\begin{proof}
Suppose, for a contradiction, that the statement is false, and let the quadruple
$(D,\mathcal{P},F,e)$ be a counterexample. Let $u$ and $v$ be the tail and head of $e$,
respectively. Define
\[
A:=\{e\}\cup\{f\in \delta_F(v) : P_f\not\subseteq F\}
\quad\text{and}\quad
A':=\bigcup_{f\in A} P_f .
\]
Since $F$ is almost $\mathcal{P}$-respecting,  $A$ is finite and thus so is  $A'$. We assume
that the counterexample is chosen so that
\[
n(D,\mathcal{P},F,e):=\lvert A'\setminus F\rvert
\]
is minimal.

To obtain a contradiction, it suffices to show that there exists an edge
$g\in A'\setminus F$ such that $F\cup\{g\}$ is a flame. Indeed, we must then have
$g\neq e$, since $(D,\mathcal{P},F,e)$ is a counterexample. Furthermore, replacing $F$ by
$F\cup\{g\}$ yields another counterexample with
\[
n(D,\mathcal{P},F\cup\{g\},e)=n(D,\mathcal{P},F,e)-1,
\]
contradicting the minimality of $n(D,\mathcal{P},F,e)$.

By Lemma~\ref{lem: char no single}, the set $T_{v,F}=:T$ contains $u$. By
Corollary~\ref{cor: insert one}, for any edge $g\in \delta_D(T)\setminus F$, the digraph
$F\cup\{g\}$ is a flame. Thus it remains to show that
\[
\delta_D(T)\cap(A'\setminus F)\neq\emptyset.
\]

Suppose, for a contradiction, that $\delta_D(T)\cap(A'\setminus F)=\emptyset$. Then for
every $f\in A$ we have $P_f\cap \delta_D(T)\subseteq F$. Note that the sets $P_f\cap \delta_D(T)=P_f\cap \delta_F(T)$ for $f\in A$
are nonempty and pairwise disjoint. Therefore,
\[
\left|\bigcup_{f\in A} P_f\cap \delta_F(T)\right|\geq |A|,
\]
and hence $|\delta_F(T)\cap A'|\geq |A|$.

Let
\[
\mathcal{Q}:=\{P_f : f\in \delta_F(v)\setminus(A\setminus\{e\})\}.
\]
On the one hand, $\mathcal{Q}$ covers all but $|A|-1$ ingoing edges of $v$ in $F$. On the
other hand, there are at least $|A|$ edges in
$\delta_F(T)$ that are not used by $\mathcal{Q}$ because $\bigcup_{f\in A} P_f\cap \delta_F(T)$ consists of such edges. This contradicts
Corollary~\ref{cor: at most n} applied to $F$  and the terminal segments of the paths in $\mathcal{Q}$ from their last edges in $\delta_{F}(T)$.
\end{proof}

Let $\prec$ be a well-order of $E$ of order type $|E|$. The special case of
Theorem~\ref{thm: main} where $E$ is countable follows from
Lemma~\ref{lem: add fin} by a straightforward recursion.

Indeed, suppose that a finite flame
\[
F_k=\{f_0,\dots,f_m\}\subseteq E
\]
has already been defined. Since finite flames are clearly $\mathcal{P}$-respecting, we
may apply Lemma~\ref{lem: add fin} to $F_k$ and the $\prec$-least element $e$ of
$E\setminus F_k$. Let $\{e_i : i\leq n\}$ be the resulting set. We define
$f_{m+i}:=e_i$ and
\[
F_{k+1}:=\{f_0,\dots,f_m,\dots,f_{m+n}\}.
\]
The recursion is complete, and the enumeration
\[
F=\{f_n : n\in\mathbb{N}\}
\]
yields the desired well-order.

Now suppose that $E$ is uncountable and let $\kappa:=|E|$. We require the following
lemma.

\begin{lemma}\label{lem: add countable}
If $F\subseteq E$ is $\mathcal{P}$-respecting and $e\in E\setminus F$, then there exists
a set $\{e_n : n\in\mathbb{N}\}\subseteq E\setminus F$ containing $e$ such that
$F\cup\{e_n : n\leq m\}$ is a flame for every $m\in\mathbb{N}$ and
$F\cup\{e_n : n\in\mathbb{N}\}$ is $\mathcal{P}$-respecting.
\end{lemma}

\begin{proof}
The lemma follows by iteratively applying Lemma~\ref{lem: add fin} together with
standard ``bookkeeping''. We recursively construct an increasing chain
$(F_n)_{n\in\mathbb{N}}$ of almost $\mathcal{P}$-respecting flames such that:
\begin{enumerate}[label=(\arabic*)]
\item $F_0=F$;
\item $e\in F_1$;
\item $F_{n+1}\setminus F_n=\{e^n_1,\dots,e^n_{k_n}\}$ is finite, and
      $F_n\cup\{e^n_1,\dots,e^n_\ell\}$ is a flame for each $\ell\leq k_n$;
\item $F^*:=\bigcup_{n\in\mathbb{N}} F_n$ is $\mathcal{P}$-respecting.
\end{enumerate}

Let $\{e^1_1,\dots,e^1_{k_1}\}$ be obtained by applying Lemma~\ref{lem: add fin} to $F$
and $e$, and set $F_1:=F_0\cup\{e^1_1,\dots,e^1_{k_1}\}$. Suppose $F_n$ is defined for
some $n\geq 1$. If $F_n$ is $\mathcal{P}$-respecting, we set $F_m:=F_n$ for all $m>n$
and stop. Otherwise, there exists an edge $f_n\in F_n$ with $P_{f_n}\not\subseteq F_n$.
Choose $f_n$ such that the unique index $m_n$ with
$f_n\in F_{m_n+1}\setminus F_{m_n}$ is minimal.

Apply Lemma~\ref{lem: add fin} to $F_n$ and an edge $g_n\in P_{f_n}\setminus F_n$, and
let $\{e^n_1,\dots,e^n_{k_n}\}$ be the resulting set. Define
$F_{n+1}:=F_n\cup\{e^n_1,\dots,e^n_{k_n}\}$. This completes the recursion.

The enumeration of $F^*\setminus F_0$ is obtained by concatenating the enumerations
$\{e^n_1,\dots,e^n_{k_n}\}$. For any $f\in F^*$, there are at most $|P_f|$ indices $n$
with $f_n=f$, and for any $m\in\mathbb{N}$ there are at most
$\bigl|\bigcup_{f\in F_{m+1}\setminus F_m} P_f\bigr|$ indices $n$ with $m_n=m$. Hence the increasing
sequence $(m_n)_{n\in\mathbb{N}}$ is unbounded.

Let $f\in F^*$ and let $m$ be such that $f\in F_{m+1}\setminus F_m$. Then there exists
$n\in\mathbb{N}$ with $m_n>m$, which implies $P_f\subseteq F_n\subseteq F^*$. Therefore
$F^*$ is $\mathcal{P}$-respecting, completing the proof.
\end{proof}

We now proceed with the proof of Theorem~\ref{thm: main} as in the countable case, using
Lemma~\ref{lem: add countable} instead of Lemma~\ref{lem: add fin}. We apply transfinite
recursion.

At a successor step, suppose that a flame $F_\alpha\subseteq E$ with
$|F_\alpha|<|E|$ and a well-order $<_\alpha$ of $F_\alpha$ have already been defined.
Apply Lemma~\ref{lem: add countable} to $F_\alpha$ and the $\prec$-least element $e$
of $E\setminus F_\alpha$. Define
\[
F_{\alpha+1}:=F_\alpha\cup\{e_n : n\in\mathbb{N}\},
\]
and let $<_{\alpha+1}$ be the end-extension of $<_\alpha$ in which the edges of
$F_{\alpha+1}\setminus F_\alpha$ are ordered according to their enumeration.

If $\alpha$ is a limit ordinal, define
\[
F_\alpha:=\bigcup_{\beta<\alpha} F_\beta
\]
and let $<_\alpha$ be the union of the compatible well-orders $<_\beta$ for
$\beta<\alpha$. This completes the transfinite recursion, and $<_\kappa$ is the desired
well-order of $E$.

\printbibliography
\end{document}